\documentclass[12pt]{amsart}

\usepackage[english]{babel}
\usepackage[utf8x]{inputenc}
\usepackage[T1]{fontenc}
\usepackage{amsmath}
\usepackage{amstext}
\usepackage{amsfonts}
\usepackage{amssymb}
\usepackage{amsthm}
\usepackage{mathtools}
\usepackage{dsfont}
\usepackage{color}
\usepackage{graphicx}
\usepackage[colorinlistoftodos]{todonotes}
\usepackage{float}
\usepackage{leftidx}
\usepackage{tikz-cd}
\usepackage[pagebackref=true]{hyperref}

\usepackage{cleveref}
\usepackage[alphabetic, initials]{amsrefs}

\allowdisplaybreaks

\newtheorem{theorem}{Theorem}[section]
\newtheorem{lemma}[theorem]{Lemma}
\newtheorem{prop}[theorem]{Proposition}
\newtheorem{cor}[theorem]{Corollary}
\newtheorem{thm}[theorem]{Theorem}
\newtheorem{lem}[theorem]{Lemma}

\newtheorem*{cor*}{Corollary}
\newtheorem*{thm*}{Theorem 2}
\newtheorem*{defn*}{Definition 1}

\newtheorem*{prop*}{Proposition 3}
\theoremstyle{definition}
\newtheorem{rem}[theorem]{Remark}

\newtheorem{defn}[theorem]{Definition}
\newtheorem{example}[theorem]{Example}

\newtheorem*{remark*}{Remark}

\newcommand{\pr}{\operatorname{Prob}}

\newcommand{\act}{\!\curvearrowright\!}

\newcommand{\sA}{\mathsf{A}}
\renewcommand{\a}{\mathfrak{a}}

\newcommand{\sB}{\mathsf{B}}
\renewcommand{\b}{\mathfrak{b}}

\newcommand{\cE}{\mathcal{E}}

\newcommand{\cH}{\mathcal{H}}

\newcommand{\cL}{\mathcal{L}}
\newcommand{\cN}{\mathcal{N}}
\newcommand{\cP}{\mathcal{P}}
\newcommand{\p}{\mathfrak{p}}

\newcommand{\s}{\mathsf{s}}

\newcommand{\cU}{\mathcal{U}}
\newcommand{\cZ}{\mathcal{Z}}

\newcommand{\bC}{{\mathbb{C}}}

\newcommand{\bZ}{{\mathbb{Z}}}

\newcommand{\sM}{\mathsf{M}}
\newcommand{\sN}{\mathsf{N}}

\newcommand{\csr}{C^{*}_{\rm red}}

\newcommand{\om}{\omega}

\newcommand{\id}{\operatorname{id}}

\newcommand{\G}{\Gamma}
\renewcommand{\L}{\Lambda}

\newcommand{\one}{\mathbf{1}}

\title[]{On invariant subalgebras of group $C^*$ and von Neumann algebras}

\author{Mehrdad Kalantar}
\address{Mehrdad Kalantar\\ University of Houston\\ USA}
\email{mkalantar@uh.edu}

\author{Nikolaos Panagopoulos}
\address{Nikolaos Panagopoulos \\ University of Houston\\ USA}
\email{npanagopoulos@uh.edu}

\thanks{MK is supported by a Simons Foundation Collaboration Grant (\# 713667).}

\begin{document}

\maketitle

\begin{abstract}
Given an irreducible lattice $\G$ in the product of higher rank simple Lie groups, we prove a co-finiteness result for the $\G$-invariant von Neumann subalgebras of the group von Neumann algebra $\cL(\G)$, and for the $\G$-invariant unital $C^*$-subalgebras of the reduced group $C^*$-algebra $C^*_{\rm red}(\G)$.

We use these results to show that: (i) every $\G$-invariant von Neumann subalgebra of $\cL(\G)$ is generated by a normal subgroup; and (ii) given a non-amenable unitary representation $\pi$ of $\G$, every $\G$-equivariant conditional expectation on $C^*_\pi(\G)$ is the canonical conditional expectation onto the $C^*$-subalgebra generated by a normal subgroup.
\end{abstract}

\section{Introduction}

Given a discrete group $\G$ and a subgroup $\L\leq \G$, we have a canonical inclusion of the reduced group $C^*$-algebras $C^*_{\rm red}(\L)\subseteq C^*_{\rm red}(\G)$, and the characteristic function $\one_\L$ of $\L$ extends to a conditional expectation $\mathcal{E}_\L: C^*_{\rm red}(\G)\to C^*_{\rm red}(\L)$. Considering the canonical action of $\G$ on $C^*_{\rm red}(\G)$ by unitary conjugations, one  can observe that $\L$ is normal in $\G$ if and only if  $C^*_{\rm red}(\L)$ is $\G$-invariant inside $C^*_{\rm red}(\G)$, or equivalently if  the conditional expectation $\mathcal{E}_\L$ is $\G$-equivariant. 
Similar observations hold for the inclusion $\cL(\L)\subset\cL(\G)$ of the corresponding group von Neumann algebras.

In this sense, $\G$-invariant unital $C^*$-subalgebras of $C^*_{\rm red}(\G)$, and $\G$-invariant von Neumann subalgebras of $\cL(\G)$ can be viewed as ``noncommutative normal subgroups'' of $\G$. Similarly, 
$\G$-equivariant conditional expectations on $C^*_{\rm red}(\G)$ or $\cL(\G)$, can be considered as  alternative noncommutative generalizations of normal subgroups of $\G$.

This note is concerned with these objects. 
We are mostly interested in the case of lattices $\G$ in higher rank semisimple groups, where some of the most prominent results (e.g., Margulis' Normal Subgroup Theorem \cite{Marg-book}, Stuck--Zimmer's rigidity theorem \cite{StuckZim},  character-rigidity \cite{Bek07},\cite{BoutHoud}, \cite{BBHP}, \cite{Pet15}) are concerned with rigidity properties of their normal subgroups, and various generalized objects. In this context, the above considerations naturally lead us to the questions of to what extent, and in what sense, those rigidity results can be extended to the noncommutative setting.

The main result of this paper is that irreducible lattices in higher rank semisimple Lie groups $\G$ have no genuine ``noncommutative normal subgroups''.
\begin{thm}\label{thm:con-exp-rig-VN}
Assume $G$ is a connected semisimple Lie group with trivial center, no non-trivial compact factors, and such that all its simple factors have real rank at least two. Let $\G< G$ be an irreducible lattice. Then, any $\G$-invariant von Neumann subalgebra of $\cL(\G)$ is of the form $\cL(\L)$ for some normal subgroup $\L\le\G$.
\end{thm}

The work of Alekseev--Brugger~\cite{AB} and Brugger's thesis~\cite{Bru}, seem to be the first systematic study of $\G$-invariant von Neumann subalgebras of $\cL(\G)$, where the authors proved that for $\G$ as in Theorem~\ref{thm:con-exp-rig-VN} (with the additional assumption of $G$ being simple), the inclusion $\sM\subset\cL(\G)$ of any non-trivial $\G$-invariant subfactor $\sM$ of $\cL(\G)$ has finite (Jones) index (see also \cite[Section 3]{CD} for some related results). 

Recall that if $\L$ is a subgroup of $\G$, then 
$\L$ has finite index in $\G$ if and only if $\cL(\L)$ has finite index in $\cL(\G)$.
Hence, their result can be considered as a noncommutative version of 
Margulis' Normal Subgroup Theorem (NST) \cite{Marg-book}, which states that every group $\G$ as in Theorem~\ref{thm:con-exp-rig-VN} is \emph{just-infinite}, meaning that every non-trivial normal subgroup of $\G$ has finite index in $\G$. 

In the $C^*$-algebra setting, we prove a complete description of equivariant conditional expectations as follows. 
We denote by $\lambda_\G$ the left regular representation of $\G$.

\begin{thm}\label{thm:con-exp-rig-C}
Let $\G$ be as in the statement of Theorem~\ref{thm:con-exp-rig-VN}, and let $\pi$ be a non-amenable unitary representation of $\G$ (in the sense of Bekka~\cite{Bek90}).
If $\cE$ is a $\G$-equivariant conditional expectation on $C^*_{\pi}(\G)$, then there is a normal subgroup $\L\le\G$ such that $\cE|_{\pi(\G)}=\one_{\pi(\L)}$.
\end{thm}
Here, by a conditional expectation on a unital $C^*$-algebra $\sB$ we mean a unital completely positive (ucp) idempotent map $\cE: \sB\to\sB$ whose range is a $C^*$-subalgebra of $\sB$, and the action of $\G$ on $C^*_{\pi}(\G)$ is by conjugations with unitaries $\pi(g)$, $g\in\G$. 

Let us point out that the amenability condition is quite restrictive for unitary representations of property (T) groups (see \cite{BekVal93}). Examples of non-amenable representations $\pi$ of groups $\G$ as in the statement of Theorem~\ref{thm:con-exp-rig-C} include, for instance, the quasi-regular representation of $\G$ on $\ell^2(\G/\L)$ for any subgroups $\L\le \G$ of infinite index, and the Koopman representation of $\G$ on $L^2(X, \nu)$ for any $\G$-quasi-invariant measure on a compact minimal $\G$-space admitting no $\G$-invariant probability measure (e.g., any topological $\G$-boundary $X$).

Since every von Neumann subalgebra of a finite von Neumann algebra is the image of a unique trace-preserving conditional expectation, Theorem~\ref{thm:con-exp-rig-VN} can be equivalently stated as a rigidity result for conditional expectations on $\cL(\G)$.
 \begin{thm}\label{thm: LG}
 Let $\G$ be as in the statement of Theorem~\ref{thm:con-exp-rig-VN}. If $\cE$ is an ultraweakly continuous $\G$-equivariant conditional expectation on $\mathcal{L}(\G)$, then there is a normal subgroup $\L\le\G$ such that $\cE|_{\lambda_{\G}(\G)}=\one_{\lambda_{\G}(\L)}$.
 \end{thm}

In contrast to the von Neumann algebra case, in general, 
even for groups $\G$ with the infinite conjugacy class (icc) property, it is not the case that any $\G$-invariant unital $C^*$-subalgebra $\sA$ of $C^*_{\rm red}(\G)$ is the image of some conditional expectation (see Example~\ref{ex:no-cond}). 
In particular, Theorem~\ref{thm:con-exp-rig-C} does not offer a complete description of all $\G$-invariant unital $C^*$-subalgebras $\sA$ of $C^*_{\rm red}(\G)$, although in the case of these lattices, we do not know of any concrete example of such $\sA$ that is not the image of an equivariant conditional expectation of $C^*_{\rm red}(\G)$.

We do however prove, in Theorem~\ref{thm:main} below, a co-finiteness result for general $\G$-invariant unital $C^*$-subalgebras of $C^*_{\rm red}(\G)$, where $\G$ is a lattice as in Theorem~\ref{thm:con-exp-rig-VN}.

The co-finiteness of $C^*$-subalgebras in our result is in the following sense.
\begin{defn} \label{def-main}
Let $\pi$ be a unitary representation of the group $\G$.

A $\G$-invariant von Neumann subalgebra $\sN$ of $\pi(\G)^{\prime\prime}$ is said to be \emph{co-finite in $\pi(\G)^{\prime\prime}$}, if the commutant ${\sN}^\prime$ of $\sN$ in $B(H_\pi)$ admits an ultraweakly continuous $\G$-invariant state (commutants are taken in $B(H_\pi)$).

A $\G$-invariant unital $C^*$-subalgebra $\sA$ of $C^*_\pi(\G)$ is said to be \emph{co-finite in $C^*_\pi(\G)$}, if the double commutant ${\sA}^{\prime\prime}$ is co-finite in $\pi(\G)^{\prime\prime}$. 
\end{defn}

\begin{theorem}\label{thm:main}
Let $\G$ be as in the statement of Theorem~\ref{thm:con-exp-rig-VN}. 
Then for every non-trivial unital $\G$-invariant $C^*$-subalgebra $\sA$ of $\csr(\G)$ and every unitary representation $\pi$ of $\G$ which is weakly contained in $\lambda_\G$, $\pi(\sA)$ is co-finite in $C^*_\pi(\G)$.
\end{theorem}
We also prove a von Neumann algebraic version of Theorem~\ref{thm:main}, recorded as Theorem~\ref{thm:main-vN}, which will be a key ingredient in the proof of Theorem~\ref{thm:con-exp-rig-VN}.

The conclusion of Theorem~\ref{thm:main} is indeed a noncommutative just-infiniteness for $\G$ (see Definition~\ref{def:c*-j-inf}). 
Obviously, any group $\G$ that satisfies the conclusion of Theorem~\ref{thm:main} is just-infinite, but not conversely: we will give examples of simple non-amenable groups $\G$ for which the conclusion of Theorem~\ref{thm:main} does not hold (see Example \ref{ex:LeB}). In this sense, the above theorem is a sharper result than Margulis' NST for the lattices $\G$ as in the statement.

We close this introduction by pointing out that a different concept of just-infiniteness for $C^*$-algebras has been studied by Grigorchuk-Musat-R\o rdam \cite{GrigMusRor}.


%

%

\section{$\G$-invariant subalgebras}

In this section we prove some general properties of co-finite $\G$-invariant subalgebras, mostly in the case of $\csr(\G)$ and $\mathcal{L}(\G)$.
Throughout the section, $\G$ is always a discrete group.

We observe that if $\L$ is a finite-index normal subgroup of $\G$, then for any unitary representation $\pi$ of $\G$, $C^*_\pi(\L)$ is co-finite in $C^*_\pi(\G)$, and $\pi(\L)^{\prime\prime}$ is co-finite in $\pi(\G)^{\prime\prime}$. Indeed, in this case, the $\G$-action on $\sM:= {C^*_\pi(\L)}^{\prime} = \pi(\L)^{\prime}$ factors through the finite group $\G/\L$, hence $\sM$ admits a $\G$-invariant ultraweakly continuous state.

One expects that for any natural notion of co-finiteness, certain properties should lift from co-finite subalgebras to the ambient space. In the following theorem, we prove that amenability and Haagerup property are among such properties. 
\begin{thm}\label{thm:no-amen-sub-pi}
Let $\pi$ be a factorial unitary representation of $\G$. 
\begin{enumerate}
\item
If $\pi(\G)^{\prime\prime}$ has a non-trivial co-finite $\G$-invariant von Neumann subalgebra with the Haagerup property, then $\pi(\G)^{\prime\prime}$ also has the Haagerup property.
\item
If $\pi(\G)^{\prime\prime}$ has an amenable non-trivial co-finite $\G$-invariant von Neumann subalgebra, then $\pi(\G)^{\prime\prime}$ is amenable.
\end{enumerate}
\begin{proof}
(1) Let $\sN\subset \pi(\G)^{\prime\prime}$ be a non-trivial co-finite $\G$-invariant von Neumann subalgebra that has the Haagerup property. Let $\sM=\sN^\prime\subset B(H_\pi)$. Then $\sM$ has the Haagerup property \cite[Theorem 3.12]{OkaTom15}.
Let $\tau$ be the $\G$-invariant ultraweakly continuous state on $\sM$ guaranteed by the co-finiteness assumption, and let $\p\in\sM$ be the support projection of $\tau$. Then $\p$ is $\G$-invariant, and therefore $\p\in\pi(\G)^\prime$. Furthermore, $\tau$ restricts to a faithful $\G$-invariant ultraweakly continuous state on the $\G$-von Neumann algebra $\p\sM\p$, which also has the Haagerup property \cite[Corollary 3.13]{OkaTom15}. Since there exists a conditional expectation $\p\sM\p\to (\p\sM\p)^\G = \p\sM^\G\p = \p\pi(\G)^\prime\p$ (see e.g., \cite[Proposition 2.7(1)]{BBHP}), it follows that $\p\pi(\G)^\prime\p$ has the Haagerup property \cite[Theorem 5.9]{OkaTom15}. Since $\p\pi(\G)^\prime\p$ is a factor, this implies $\pi(\G)^\prime$ has the Haagerup property \cite[Corollary 3.13]{OkaTom15}. Invoking \cite[Theorem 3.12]{OkaTom15} once again, the result follows.\\[0.8ex]
(2) The proof is similar, the corresponding permanence results regarding amenability are well-known (see e.g., \cite{DelPop-book}).
\end{proof}
\end{thm}
\begin{cor}\label{cor:amen-cofin->amen}
Let $\G$ be icc. Then
\begin{enumerate}
\item
$\G$ is amenable if and only if  $\cL(\G)$ has a non-trivial co-finite $\G$-invariant amenable von Neumann subalgebra.
\item
$\G$ has the Haagerup property if and only if $\cL(\G)$ has a non-trivial co-finite $\G$-invariant von Neumann subalgebra that has the Haagerup property.
\end {enumerate}
\begin{proof}
The arguments for both parts are similar and therefore we only give the proof of (2). If $\G$ has the Haagerup property then $\cL(\G)$ has the Haagerup property, and so the forward direction follows. Conversely, since $\G$ is icc, the regular representation of $\G$ is factorial. If $\cL(\G)$ has a non-trivial co-finite $\G$-invariant von Neumann subalgebra that has the Haagerup property, Theorem~\ref{thm:no-amen-sub-pi} implies that $\cL(\G)$ has the Haagerup property. Hence $\G$ has the Haagerup property \cite[Theorem~3]{Chod83}.
\end{proof}
\end{cor}

\subsection{Noncommutative just-infiniteness}
Any notion of co-finiteness in the operator algebraic setting naturally leads to noncommutative notions of just-infiniteness. 
\begin{defn} \label{def-main}
We say $\G$ is \emph{$\pi$-just-infinite} if $\pi$ is infinite-dimensional and every nontrivial $\G$-invariant von Neumann subalgebra $\sN\subset \pi(\G)^{\prime\prime}$ is co-finite in $\pi(\G)^{\prime\prime}$. 
\end{defn}
 Recall that $\lambda_\G$ denotes the left regular representation of $\G$.
\begin{lem}\label{lem:icc}
If $\G$ is $\lambda_\G$-just-infinite, then $\G$ is icc. 
\begin{proof}
Assume $\G$ is not icc. Then $\cL(\G)$ contains a non-trivial central projection $\p$.  
Then, $\sN= \bC \p \oplus \bC \p^\perp$ is a non-trivial $\G$-invariant von Neumann subalgebra of $\cL(\G)$, and $\sN^\prime= \p B(\ell^2(\G))\p\,\oplus\, \p^\perp B(\ell^2(\G))\p^\perp$. Since $\G$ is $\lambda_\G$-just-infinite, $\sN^\prime$ admits a $\G$-invariant ultraweakly continuous state. The map $\a \mapsto \p\a \p + \p^\perp \a \p^\perp$, from $B(\ell^2(\G))$ to $\sN^\prime$, is ultraweakly continuous, $\G$-equivariant, and ucp. Hence, we obtain a $\G$-invariant ultraweakly continuous state on $B(\ell^2(\G))$, which implies that $\G$ is finite. This is a contradiction as $\lambda_\G$-just-infinite groups are by definition assumed to be infinite.
\end{proof}
\end{lem}

\begin{example}\label{ex:Z}
It follows from Lemma~\ref{lem:icc} that
there are just-infinite groups that are not $\lambda_\G$-just-infinite.
%
Indeed, the integer group $\G=\bZ$ is just-infinite (in fact, it even satisfies the conclusion of Stuck--Zimmer's rigidity theorem), but not icc, hence it is not $\lambda_\G$-just-infinite by Lemma~\ref{lem:icc}.
\end{example}

\begin{prop}\label{prop:no-Haag-Vn}
Assume $\G$ is $\lambda_\G$-just-infinite and does not have the Haagerup property. Then $\cL(\G)$ does not admit any non-trivial $\G$-invariant von Neumann subalgebras with the Haagerup property.
\begin{proof}
For the sake of contradiction assume
that $\sN$ is a non-trivial $\G$-invariant von Neumann subalgebra of $\cL(\G)$ with the Haagerup property. Since, $\G$ is $\lambda_\G$-just-infinite, $\sN$ is co-finite in $\cL(\G)$, hence Lemma~\ref{lem:icc} and Corollary~\ref{cor:amen-cofin->amen} entail that $\cL(\G)$ has the Haagerup property. Therefore, by \cite[Theorem~3]{Chod83}, the group $\Gamma$ has the Haagerup property, which contradicts our assumption.
\end{proof}
\end{prop}

\begin{prop}\label{prop:no-amen-Vn}
Assume that $\G$ is $\lambda_\G$-just-infinite and non-amenable. Then $\cL(\G)$ does not admit any non-trivial $\G$-invariant amenable von Neumann subalgebras.
\begin{proof}
The argument is similar to the proof of Proposition~\ref{prop:no-Haag-Vn}, using the fact that $\cL(\G)$ is amenable if and only if $\G$ is amenable.
\end{proof}
\end{prop}
In particular, we conclude
\begin{cor}\label{cor:no-amena}
Assume that $\G$ is $\lambda_\G$-just-infinite and non-amenable. Then $\cL(\G)$ does not admit any Cartan subalgebra with $\lambda_\G(\G)\subseteq\cN_{\cL(\G)}(\sN)$.
\end{cor}

Let us point out that the current understanding of Cartan subalgebras of $\mathcal{L}(\G)$ is far from complete in the higher-rank realm. Important partial results in this direction were obtained recently by Boutonnet--Ioana--Peterson in \cite{BouIoaPet}, where they proved the lack of certain classes of Cartan subalgebras.

We conclude the section with another useful consequence of Proposition~\ref{prop:no-amen-Vn}.

\begin{cor}\label{cor:no-amena}
Assume $\G$ is $\lambda_\G$-just-infinite and non-amenable. Then every $\G$-invariant von Neumann subalgebra of $\cL(\G)$ is a factor.
\begin{proof}
Let $\sM$ be a $\G$-invariant von Neumann subalgebra of $\cL(\G)$. Then the center $\cZ(\sM)$ of $\sM$ is a $\G$-invariant abelian von Neumann subalgebra of $\cL(\G)$. Hence, $\cZ(\sM)=\bC1$ by Proposition~\ref{prop:no-amen-Vn}.
\end{proof}
\end{cor}

\subsection*{Weak just-infiniteness}

We have parallel results in the $C^*$-algebra setup, and as we see, a formally weaker just-infiniteness condition would suffice.

\begin{defn} \label{def-main}
Let $\pi$ be a unitary representation of $\G$. 
We say $\G$ is \emph{weakly $\pi$-just-infinite} if $\pi$ is infinite-dimensional and every nontrivial $\G$-invariant unital $C^*$-subalgebra of $C^*_\pi(\G)$ is co-finite in $C^*_\pi(\G)$.
\end{defn}
The next two results are $C^*$-algebraic analogues of Propositions~\ref{prop:no-Haag-Vn} and~\ref{prop:no-amen-Vn}. By the canonical trace, we mean the (unique) state on $\csr(\G)$ that extends the dirac function supported on the neutral element $e\in\G$.
\begin{prop}\label{prop:no-Haag-C*}
Assume $\G$ is weakly $\lambda_\G$-just-infinite and does not have the Haagerup property. Then $\csr(\G)$ does not admit any non-trivial $\G$-invariant exact $C^*$-subalgebra on which the restriction of the canonical trace is an amenable trace.
\begin{proof}
For the sake of contradiction, assume that $\sA$ is a non-trivial exact $\G$-invariant unital $C^*$-subalgebra of $\csr(\G)$ such that the restriction of the canonical trace $\tau$ to $\sA$ is an amenable trace on $\sA$. Then by \cite[Corollary 3.7]{Suz13} the pair $(\sA, \tau)$ has the Haagerup property in the sense of \cite{Dong11} and so from \cite[Lemma 4.5]{Suz13} follows that $\sA^{\prime\prime}$ has the Haagerup property. Since $\sA$ is co-finite in $\csr(\G)$ by $\lambda_\G$-just-infiniteness, $\sA^{\prime\prime}$ is co-finite in $\cL(\G)$ and so as in the proof of Proposition~\ref{prop:no-Haag-Vn} we get a contradiction.
\end{proof}
\end{prop}

\begin{prop}\label{prop:no-amen-C*}
Assume that $\G$ is weakly $\lambda_\G$-just-infinite and non-amenable. Then $\csr(\G)$ does not admit any non-trivial $\G$-invariant nuclear $C^*$-subalgebras.
\begin{proof}
For the sake of contradiction assume that $\csr(\G)$ admits a non-trivial nuclear $\G$-invariant unital $C^*$-subalgebra $\sA$. Then $\sA$ is co-finite in $\csr(\G)$ by $C^*$-just-infiniteness and so $\sA^{\prime\prime}$ is a non-trivial amenable co-finite von Neumann subalgebra of $\cL(\G)$. Corollary~\ref{cor:amen-cofin->amen} implies that $\cL(\G)$ is amenable, hence $\G$ is amenable. This is a contradiction, and the proof is complete.
\end{proof}
\end{prop}

%

\subsection{$C^*$-just-infinite groups}
In this section we prove Theorem~\ref{thm:main}. 

\begin{defn} \label{def:c*-j-inf}
A discrete group $\G$ is said to be \emph{$C^*$-just-infinite}, if it is infinite and for every non-trivial invariant unital $C^*$-subalgebra $\sA$ of $\csr(\G)$ and every unitary representation $\pi$ of $\G$ which is weakly contained in $\lambda_\G$, $\pi(\sA)$ is co-finite in $C^*_\pi(\G)$.
\end{defn}

We remark that for a given unitary representation $\pi$ of $\G$ and an invariant $C^{*}$-subalgebra $\sA\subset C^*_\pi(\G)$, co-finiteness of $\sA$ is effectively a von Neumann algebraic property: it is equivalent to the same property for the bi-commutant ${\sA}^{\prime\prime}$.
Noting that the unitary representations $\pi$ of $\G$ which are weakly contained in $\lambda_\G$ correspond to $\ast$-representations of $\csr(\G)$, we see that $C^*$-just-infiniteness is defined in terms of von Neumann algebras generated by all $\ast$-representations of $\csr(\G)$.

Before getting to the proof of Theorem~\ref{thm:main}, let us prove some general properties of $C^*$-just-infinite groups.
\begin{prop} \label{prop:counter}
Let $\G$ be a $C^*$-just-infinite group. Then 
$\G$ is either amenable or $C^{*}$-simple.
\begin{proof}
Assume that $\G$ is not $C^*$-simple. Let $\pi\colon  \csr(\G)\to B(H_\pi)$ be a non-faithful representation, and let ${I}={\rm ker}(\pi)$. In particular, $\pi$ defines a unitary representation of $\G$ which is weakly contained in $\lambda$. Define $\sA = {I} + \bC 1$. Then $\sA$ is a $\G$-invariant unital $C^*$-subalgebra of $\csr(\G)$, and $\pi(\sA)= \bC 1$. Thus, by the assumption, there exists a $\G$-invariant state on $B(H_\pi)$. Since $\pi$ is weakly contained in $\lambda$, there is also a $\G$-invariant state on $B(\ell^2(\G))$ \cite[Corollary~5.3]{Bek90}. Restricting this state to $\ell^\infty(\G)$, it yields an invariant mean on $\ell^\infty(\G)$, and so $\G$ is amenable.
\end{proof}
\end{prop}

\begin{example}\label{ex:LeB}
It follows from Proposition~\ref{prop:counter} that
there are simple non-amenable groups $\G$ that are not $C^*$-just-infinite.
Indeed, the existence of non-$C^*$-simple groups with the unique trace property was proven by Le Boudec in \cite{LBou17}. And as remarked after \cite[Theorem~C]{LBou17}, those examples include groups admitting simple subgroups of index two. Obviously any such subgroup is also non-$C^*$-simple, and non-amenable, and therefore also non-$C^*$-just-infinite by Proposition~\ref{prop:counter}.
\end{example}


\subsection*{Proof of Theorem~\ref{thm:main}}
The proof of Theorem~\ref{thm:main} is structured similarly to Margulis' proof of the  Normal Subgroup Theorem, consisting of two parts: the Amenability half, and the Property~(T) half. 
\begin{proof}[Proof of Theorem~\ref{thm:main}]
Let $\sA\subset \csr(\G)$ be a non-trivial $\G$-invariant unital $C^*$-subalgebra, and let $\pi$ be a unitary representation of $\G$ which is weakly contained in $\lambda$. 
\\[1.2ex]
{\bf Amenability half.}\, We prove that ${\pi(\sA)}^\prime$ admits a $\G$-invariant state.

Let $P$ be a minimal parabolic subgroup, and $K$ a maximal compact subgroup of $G$ such that $G=KP$. Let $\nu_P$ be the unique $K$-invariant probability on $G/P$, and $\mu$ a fully supported probability on $\G$ such that $(G/P, \nu_P)$ is the Poisson boundary of the $(\G, \mu)$-random walk \cite{Furs67}.

The action $G\act G/P$ is faithful; indeed, since ${\rm ker}(G\act G/P)$ is contained in $P$, it is a normal amenable subgroup of $G$, which is also a subproduct of $G$ by semisimplicity. Since each factor of $G$ is non-amenable, it follows ${\rm ker}(G\act G/P)$ is trivial. In particular, $\G$ also acts faithfully on $G/P$.
Thus, since the action $\G\act G/P$ is strongly proximal \cite{Furs67}, it follows $\G$ has trivial amenable radical \cite[Proposition~7]{Furm03}. 
This implies $\G$ is $C^*$-simple by \cite[Theorem~1]{BCdlH} (cf. \cite[Corollary~6.10]{BKKO}). Hence, $\pi$ is weakly equivalent to $\lambda$, and therefore the map $\pi(g)\mapsto\lambda(g)$ extends to a $\G$-equivariant ucp map $\psi\colon  B(H_\pi)\to B(\ell^2(\G))$. 
Since $C^*_\pi(\G)$ lies in the multiplicative domain of $\psi$, it follows that $\psi({\pi(\sA)}^\prime)\subset {\sA}^\prime$. Therefore, it suffices to show that ${\sA}^\prime$ admits a $\G$-invariant state.

Define $\theta_\mu=\sum_{g\in\G}\mu(g) {\rm Ad}_{\rho(g)}$, where $\rho\colon \G\to B(\ell^2(\G))$ is the right regular representation. Then $\theta_\mu$ is a $\G$-equivariant ucp map on $B(\ell^2(\G))$, and we denote by ${\widetilde{\cH}}_\mu$ the noncommutative Poisson boundary of $(\G,\mu)$ in the sense of Izumi \cite{Izu04}, i.e., the space $\{\a\in B(\ell^2(\G)) : \theta_\mu(\a) = \a\}$ of fixed points of $\theta_\mu$. We also denote by $\cH_\mu = {\widetilde{\cH}}_\mu\cap \ell^\infty(\G)$  the space of bounded $\mu$-harmonic functions on $\G$.

Any point-weak* limit $\cE_\mu$ of the sequence of maps $\frac1n \sum_{k=1}^{n} \theta_{\mu^{k}}$ defines a $\G$-equivariant ucp idempotent from $B(\ell^2(\G))$ onto ${\widetilde{\cH}}_\mu$, and endowed with the Choi--Effros product defined by $\cE_\mu$, the space ${\widetilde{\cH}}_\mu$ becomes a von Neumann algebras. Note that ${\widetilde{\cH}}_\mu$ is also canonically equipped with a $\G$-action since it is a $\G$-invariant subspace of $B(\ell^2(\G))$. 

Since $\csr(\G)$ lies in the multiplicative domain of $\cE_\mu$, it follows that $\cE_\mu({\sA}^\prime)\subset {\sA}^\prime\cap {\widetilde{\cH}}_\mu$. We will show that the latter intersection admits a $\G$-invariant state, hence its composition with $\cE_\mu$ yields a $\G$-invariant state on ${\sA}^\prime$, and this will conclude the proof of this part.

By \cite[Theorem~4.1]{Izu04}, there is a $\G$-equivariant von Neumann algebra isomorphism $\Phi\colon  \G\ltimes L^\infty (G/P)\to {\widetilde{\cH}}_\mu$ that restricts to the identity map on $\csr(\G)$, and to the Poisson transform $\cP_{\nu_P}\colon  L^\infty (G/P) \to \cH_\mu$ defined by $\cP_{\nu_P}(f)(g) = \int_{G/P}f(gx) \,d\nu_P(x)$ for $f\in L^\infty (G/P)$ and $g\in \G$. 

Since $\Phi$ is identity on $\csr(\G)$, it maps the relative commutant $\sB$ of $\sA$ in the crossed product $\G\ltimes L^\infty (G/P)$ onto ${\sA}^\prime\cap {\widetilde{\cH}}_\mu$. Thus, it now suffices to show that the $\G$-von Neumann algebra $\sB$ admits a $\G$-invariant state. 

Let $\cE_0\colon \G\ltimes L^\infty (G/P)\to L^\infty (G/P)$ be the canonical conditional expectation, which is $\G$-equivariant and faithful. Then, the composition $\varrho(\cdot) \colon = \int_{G/P} \cE_0(\cdot)(x) \,d\nu_P(x)$ defines a faithful ultraweakly continuous $\mu$-stationary state on $\G\ltimes L^\infty (G/P)$. We will show that the restriction of $\varrho$ to $\sB$ is $\G$-invariant. 

First, we see that $\G\act \sB$ is ergodic, i.e., the $\G$-invariant subalgebra $\{\b\in \sB : g\b=\b ~ \forall g\in \G\}$ of $\sB$ is trivial. One can show this, for instance, by using the freeness of $\G\act G/P$ to conclude a unique stationarity as in \cite[Example~4.13]{HartKal}, and then apply \cite[Proposition~2.7.(3)]{BBHP}; (alternatively, \cite[Theorem A]{DasPet20} can be used to show the claim).
However, it was pointed out to us by R\'{e}mi Boutonnet that the ergodicity in the above follows from the following general fact, which we include for the future easy reference; we thank him for providing us with the proof.
%
%
\begin{lemma}[Boutonnet]\label{lem:bout}
Let $\G$ be an icc group  with $\mu \in \pr(\G)$ whose support generates $\G$. If $(X,\nu)$ is
any ergodic $(\G, \mu)$-space, then the action of $\G$ on the crossed product $\G\ltimes L^\infty (X)$ by conjugation is ergodic.
\begin{proof}
Denote by $\cE: \G\ltimes L^\infty (X)\to L^\infty (X)$ the canonical conditional expectation.
Let $\a\in \G\ltimes L^\infty (X)$ be $\G$-fixed, and set $\b=\a- \cE(\a)$. Then $\b$ is $\G$-fixed and $\cE(\b)=0$. 
For $g \in \G$, let $\phi_{g} := \cE(\b \lambda_{g^{-1}}) \in L^{\infty}(X)$ be the Fourier coefficient of $\b$ at $g$. 
Since $\b$ is $\G$-fixed, the map $g\mapsto \left|\phi_{g}\right|^{2}$ is $\G$-equivariant with respect to the conjugation action of $\G$ on itself, and so the function $\zeta(g):= \int_{X} \left|\phi_{g}\right|^{2}d\nu$ is $\mu$-harmonic (for the conjugate action).  Moreover, $\sum_{g \in \G} \zeta(g) = \int_{X} \cE(\b^{*} \b)d\nu<\infty$, and so $\zeta\in\ell^{1}(\G)$. Thus, $\zeta$ attains its maximum values on a finite conjugate-invariant subset of $\G$, that is $\{e\}$ by the icc assumption. Since $\zeta\geq 0$ and $\zeta(e) = 0$, it follows $\zeta = 0$. So $\b=0$, and $\a=\cE(\a)$. The function $\cE(\a)\in L^\infty(X)$ is $\G$-invariant, hence constant by ergodicity of $\G\act (X, \nu)$.
\end{proof}
\end{lemma}
Returning to the proof of the theorem, we can now invoke \cite[Theorem~B]{BoutHoud} to conclude that either the restriction ${\varrho}|_\sB$ is $\G$-invariant, or that there exists a closed subgroup $P\subset Q\subsetneq G$ and a $\G$-equivariant von Neumann algebra embedding  $\vartheta\colon L^{\infty}(G/Q,\nu_{Q}) \to \sB$, such that $\varrho \circ \vartheta=\nu_{Q}$, where $\nu_Q$ is the pushforward of $\nu_P$ under the canonical map $G/P\to G/Q$. 
We will rule out the latter possibility.

For the sake of contradiction, assume otherwise, that such $\vartheta$ exists. Then the composition $\cE_0\circ \vartheta$ is an ultraweakly continuous $\G$-equivariant ucp map from $L^{\infty}(G/Q,\nu_{Q})$ into $L^{\infty}(G/P,\nu_P)$. The canonical embedding is the unique such map (see e.g., \cite[Theorem~2.14]{BadShal06}), so $\cE_0\circ \vartheta=\id|_{L^{\infty}(G/Q)}$. Since $\cE_0$ is a faithful conditional expectation, it follows $\vartheta=\id|_{L^{\infty}(G/Q)}$ by \cite[Lemma~3.3]{Ham85}.
In particular, $L^{\infty}(G/Q)\subset \sB\cap L^{\infty}(G/P)$, and therefore, $\Phi\left(L^{\infty}(G/Q)\right)\subset \sA^{\prime}\cap \cH_\mu$.
As remarked above, $\Phi$ restricts to the Poisson transform on $L^\infty(G/P)$, so $\sA\subset {\left(\cP_{\nu_Q}(L^{\infty}(G/Q))\right)}^\prime$. 

Since $\sA$ is non-trivial, and $\sA\cap \ell^\infty(\G)=\bC1$, there exists $\a\in \sA$ and $g, h\in \G$ with $g\neq h$ such that $\langle \a\delta_{g},\delta_{h}\rangle\neq 0$. 
Denoting by $\mathfrak{m}_\phi\in B(\ell^2(\G))$ the multiplication operator associated to a function $\phi\in \ell^\infty(\G)$, $\a$ commutes with $\rho(g)\mathfrak{m}_{\cP_{\nu_Q}(f)}\rho(g^{-1}) = \mathfrak{m}_{\cP_{g\nu_Q}(f)}$ for every $g\in \G$ and $f\in C(G/Q)$. The action $G\act G/Q$ is strongly proximal, so given $x\in G/Q$ there exits a net $(g_i)$ in $\G$ such that $g_i\nu_Q\xrightarrow{\text{weak*}} \delta_x$ in ${\rm Prob}(G/Q)$. This implies $\mathfrak{m}_{\cP_{g_i\nu_Q}(f)} \xrightarrow{\rm WOT} \mathfrak{m}_{\cP_{\delta_x}(f)}$ in $B(\ell^2(\G))$, and therefore $\a$ commutes with $\mathfrak{m}_{\cP_{\delta_x}(f)}$ for every $f\in C(G/Q)$ and $x\in G/Q$. 
Hence, 
\[\begin{split}
 f(g x)  \langle \a \delta_{g},\delta_{h}\rangle
 &=\langle \a \mathfrak{m}_{\cP_{\delta_x}(f)}\delta_{g},\delta_{h}\rangle
 =\langle \mathfrak{m}_{\cP_{\delta_x}(f)}\a \delta_{g},\delta_{h}\rangle
\\& =\langle \a \delta_{g},\mathfrak{m}_{\cP_{\delta_x}(\bar{f})}\delta_{h}\rangle=
f(h x)\langle \a \delta_{g},\delta_{h}\rangle
\end{split}\]
for all $f\in C(G/Q)$ and $x\in G/Q$, which implies that $gh^{-1}$ lies in ${\rm ker}(\G\act G/Q)$. But it follows from similar reasoning as in the beginning of the proof that $\G\act G/Q$ is faithful. This contradicts $g\neq h$, and therefore completes the proof of the first half.
\\[1.5ex]
{\bf Property~(T) half.}\, 
The assumptions imply that $G$ has property~(T) and since $\G$ is a lattice in $G$, $\G$ also has property~(T) (see e.g., \cite[Theorem~1.7.1]{BdlHV}).
Recall that $\pi$ is a unitary representation of $\G$ weakly contained in $\lambda$. 

Let $M={\pi(\sA)}^\prime$, and let $(M, H_\s, \mathsf{J}_\s, \mathsf{P}_\s)$ be its standard form in the sense of \cite{Haag75}.
By \cite[Theorem~3.2]{Haag75}, the action $\G\act M$ is implemented by a unitary representation $\sigma$ of $\G$ on $H_\s$. As usual, we endow $B(H_\s)$ with the $\G$-action defined by unitaries $\sigma_g$, $g\in \G$.

By the Amenability half, there exists a $\G$-invariant state on $M$. Since $M$ has separable predual, a standard argument yields a sequence $(\omega_n)$ of ultraweakly continuous states on $M$ such that $\|\omega_n\circ\sigma_g - \omega_n\|=\|g\omega_n - \omega_n\|\xrightarrow{n\to\infty} 0$ for every $g\in\G$.

By \cite[Lemma~2.10]{Haag75}, the map $\xi\mapsto \om_\xi|_M$ is a $\G$-equivariant homeomorphism from $\mathsf{P}_\s$ onto $M_*^+$, where $\om_\xi$ is the vector state defined by the vector $\xi$. Hence, there exists an almost $\sigma$-invariant sequence in $H_\s$, and property~(T) implies the existence of a $\sigma(\G)$-invariant unit vector $\xi_0$ in $H_\s$. The vector state $\om_{\xi_0}$ restricts to an ultraweakly continuous $\G$-invariant state on $M$. 
This completes the proof.
\end{proof}
%
%
We note that the proof of Theorem~\ref{thm:main} also yields the following:
\begin{thm}\label{thm:main-vN}
Every group $\G$ as in the statement of Theorem~\ref{thm:main} is $\lambda_\G$-just-infinite. 
\end{thm}

\begin{rem}
Margulis' NST and its generalizations such as Stuck--Zimmer's theorem \cite{StuckZim}, are considered as rigidity results for $\G$. However, there are just-infinite groups that should not really be considered as rigid objects in the same spirit as higher rank lattices. 

Examples~\ref{ex:Z} and~\ref{ex:LeB} show that the noncommutative just-infiniteness properties we consider in this section are indeed stronger properties than the usual group-theoretic notion, and therefore in this sense, Theorems~\ref{thm:main} and~\ref{thm:main-vN} entail sharper results compared to the NST for lattices $\G$ as in the statements.
\end{rem}
\begin{rem}
Another fact that follows from the arguments in the proof of Theorem~\ref{thm:main} is that, given a discrete property (T) group $\G$ and a unitary representations $\pi$ of $\G$, the weak $\pi$-just-infiniteness of $\G$ depends only on the weak equivalence class of $\pi$.
\end{rem}
%


\section{Proofs of Theorems \ref{thm:con-exp-rig-VN}, \ref{thm:con-exp-rig-C} and \ref{thm: LG}} \label{sec3}

\begin{proof}[Proof of Theorem~\ref{thm:con-exp-rig-VN}]
Let $\sN$ be a non-trivial $\G$-invariant von Neumann subalgebra of $\cL(\G)$. By Corollary~\ref{cor:no-amena} and Theorem~\ref{thm:main-vN}, $\sN$ is a factor.
As shown in \cite[Lemma~4.3.2 \& Remark~4.2.8]{Bru}, there is a non-trivial normal subgroup $\L\le\G$ such that the restriction of the $\G$-action on $\sN$ to $\L$ is inner. This gives a projective unitary representation $\sigma:\L\to \cU(\sN)$. Since $\sN$ is a factor, we also get a projective unitary representation $\sigma^\prime:\L\to \cU(\sN^\prime\cap\cL(\G))$ such that $\lambda_\G(g) = \sigma(g)\sigma^\prime(g)$ for all $g\in\L$. Furthermore, the latter decomposition is unique. In particular, both $\sigma(\L)^{\prime\prime}$ and $\sigma^\prime(\L)^{\prime\prime}$ are $\G$-invariant, hence factors by Corollary~\ref{cor:no-amena}.

Let $\tilde\sN$ be the von Neumann algebra generated by $\sigma(\L)\cup\sigma^\prime(\L)$. Then $\cL(\L)\subseteq\tilde\sN\subseteq\cL(\G)$. Since $\L$ has finite index in $\G$ (by Margulis' NST), it follows that $\cL(\L)$ has finite index in $\tilde\sN$.
Thus, as observed in the proof of \cite[Lemma~4.3.6]{Bru}, it cannot be the case that the trace vanishes on $\sigma(g)$ and $\sigma^\prime(g)$ for all non-neutral $g\in \L$.

The proof of \cite[Theorem C]{BoutHoud} can be modified to conclude that if $\pi$ is a projective unitary representation of $\L$ generating a finite factor $\sM$, then either $\sM$ is finite dimensional or else, $\tau(\pi(g)) = 0$ for all $g\in \L$, $g\neq e$, where $\tau$ is the trace of $\sM$.

Thus, either $\sigma(\L)^{\prime\prime}$ or $\sigma^\prime(\L)^{\prime\prime}$ is finite dimensional, hence trivial by Proposition~\ref{prop:no-amen-Vn} since they are $\G$-invariant.
We therefore conclude that either $\sN$ or its relative commutant $\sN^\prime\cap\cL(\G)$ contains $\cL(\L)$, hence has finite index in $\cL(\G)$. This implies that one of the these von Neumann algebras is finite dimensional \cite[Corollary~2.2.3]{Jon83}, hence trivial since they are $\G$-invariant.
Since we assumed $\sN$ is non-trivial, it follows that $\sN^\prime\cap\cL(\G)$ is trivial. 
We may now invoke \cite[Theorem 3.15]{CD} to conclude the result.
\end{proof}

\begin{proof}[Proof of Theorem~\ref{thm: LG}]
Let $\mathcal{E}$ be an ultraweakly continuous $\G$-equivariant conditional expectation on $\cL(\G)$. Then the image of $\cE$ is a $\G$-invariant von Neumann subalgebra of $\cL(\G)$, hence of the form $\cL(\L)$ for some normal subgroup $\L\le\G$, by Theorem~\ref{thm:con-exp-rig-VN}. 
Since $\G$ is icc, the canonical trace $\tau$ is the unique trace on $\cL(\G)$. Since the characteristic function $\one_\L$ extends to the unique trace preserving conditional expectation from $\cL(\G)$ onto $\cL(\L)$, it suffices to show that $\cE$ is trace preserving. For this, we observe that since $\cE$ is $\G$-equivariant, $\tau \mathcal{E}$ is $\G$-invariant and unltraweakly continuous, hence a trace on $\cL(\G)$. By the uniqueness of $\tau$, we conclude that $\tau \mathcal{E}=\tau$, thereby completing the proof.
\end{proof}


\begin{proof}[Proof of Theorem~\ref{thm:con-exp-rig-C}]
Let $\pi$ be a non-amenable unitary representation of $\G$. Then $\pi$ does not contain any finite-dimensional subrepresentation \cite[Theorem 1.3]{Bek90}. Hence, it follows from \cite[Corollary D]{BoutHoud} that $\pi$ weakly contains $\lambda_\G$, and $C^*_\pi(\G)$ has a unique trace $\tau$, which is the lift of the canonical trace on $\csr(\G)$. Let ${I}$ be the kernel of the quotient map $C^*_\pi(\G)\to \csr(\G)$. Since $\tau$ factors to a faithful trace on $\csr(\G)$, it follows ${I} = \{\mathfrak{a}\in C^*_\pi(\G) \colon \tau(\mathfrak{a}^{*} \mathfrak{a}) = 0\}$.

Assume $\cE$ is a $\G$-equivariant conditional expectation from $C_{\pi}^{*}(\G)$ onto a $C^*$-subalgebra $\sA\subseteq C_{\pi}^{*}(\G)$. Since $\cE$ is $\G$-equivariant, $\sA$ is $\G$-invariant.

Furthermore, the $\G$-equivariance of $\cE$ implies that $\tau\cE$ is $\G$-invariant, hence a trace on $C^*_\pi(\G)$. Since $\tau$ is the unique trace on $C^*_\pi(\G)$, we conclude $\tau\cE=\tau$.

Given $\mathfrak{a}\in {I}$, we get $\tau(\cE(\mathfrak{a})^*\cE(\mathfrak{a}))\le \tau(\cE(\mathfrak{a}^{*}\mathfrak{a})) = \tau(\mathfrak{a}^{*}\mathfrak{a}) = 0$, which implies $\cE({I})\subset {I}$. Thus, $\cE$ factors to a map on $\csr(\G)$, which is also a $\G$-equivariant conditional expectation, we still denote it by $\cE$. Since $\cE$ is trace-preserving, it is faithful on $\csr(\G)$, hence $\cE$ extends to an ultraweakly continuous trace-preserving conditional expectation $\cE:\cL(\G)\to \sA^{\prime\prime}$ (see \cite[Lemma~4.5]{Po}). By Theorem~\ref{thm: LG}, we have that $\cE|_{\lambda_\G(\G)}= \one_{\lambda_\G(\L)}$ for some normal subgroup $\Lambda\leq \G$. This implies $\cE|_{\pi(\G)}=\one_{\pi(\L)}$.
\end{proof}

We conclude this section with an example of an icc group $\G$ for which $\csr(\G)$ contains a $\G$-invariant unital $C^*$-subalgebra that is not the image of any $\G$-equivariant conditional expectation.

\begin{example}\label{ex:no-cond}
Let $\G$ be as in Example~\ref{ex:LeB}, and let $\sA$ be the $C^*$-subalgebra of $\csr(\G)$ constructed from a non-zero proper ideal $I$ as in the proof of Proposition~\ref{prop:counter}. Assume $\cE:\csr(\G)\to \sA$ is a $\G$-equivariant conditional expectation. Since $\csr(\G)$ has a unique trace, it follows, as seen in the proof of Theorem~\ref{thm:main} above, that $\cE$ is faithful and trace preserving, hence extends to the canonical trace preserving conditional expectation from $\cL(\G)$ onto $\sA^{\prime\prime}$. Since $\cL(\G)$ is a factor, and $I$ is non-zero, it follows $\sA^{\prime\prime}=\cL(\G)$. In particular, $\cE=\id$, and therefore $\sA=\csr(\G)$. This implies the quotient of $\csr(\G)$ by $I$ has dimension one, and so admits a trace, which contradicts the uniqueness of the canonical trace on $\csr(\G)$. Hence, $\sA$ cannot be the image of any $\G$-equivariant conditional expectation on $\csr(\G)$.  
\end{example}


\subsection*{Acknowledgements}
We thank R\'{e}mi Boutonnet and Eduardo Scarparo for their valuable comments on the earlier versions of this paper.

\end{document}